\documentclass{svjour3}                     
\smartqed  
\usepackage{bbm}
\usepackage{mathrsfs}
\usepackage{amsfonts}
\usepackage{stmaryrd}
\usepackage{graphicx}
\usepackage{subfigure}

\usepackage{bm}
\usepackage{amsmath}
\usepackage{amssymb}

\usepackage{amstext}
\usepackage{amsbsy}
\usepackage{setspace}
\usepackage{algorithm}
\usepackage{algorithmic}
\usepackage[figuresright]{rotating}
\usepackage{epstopdf}
\usepackage{makecell}

\numberwithin{equation}{section}
\numberwithin{theorem}{section}
\numberwithin{lemma}{section}
\numberwithin{remark}{section}
\usepackage{mathptmx}
\renewcommand{\figurename}{Fig.}

\newtheorem{thm}{Theorem}

\newtheorem{prop}[thm]{Proposition}
\newtheorem{cor}[thm]{Corollary}

\usepackage{mathptmx}      
\usepackage[justification=centering]{caption}

\begin{document}

\title{First exit and Dirichlet problem for the nonisotropic tempered $\alpha$-stable processes}


\author{Xing Liu and Weihua Deng
}


\institute{
Xing Liu \at
              School of Mathematics and Statistics, Gansu Key Laboratory of Applied Mathematics and Complex
Systems, Lanzhou University, Lanzhou 730000, People's Republic of China. \email{2718826413@qq.com}
\and
Weihua Deng\at School of Mathematics and Statistics, Gansu Key Laboratory of Applied Mathematics and Complex
Systems, Lanzhou University, Lanzhou 730000, People's Republic of China.
\email{dengwh@lzu.edu.cn}
}
\maketitle

\begin{abstract}
	
This paper discusses the first exit and Dirichlet problems of the nonisotropic tempered $\alpha$-stable process $X_t$. The upper bounds of all moments of the first exit position $\left|X_{\tau_D}\right|$ and the first exit time $\tau_D$ are firstly obtained. It is found that the probability density function of $\left|X_{\tau_D}\right|$ or $\tau_D$ exponentially decays with the increase of $\left|X_{\tau_D}\right|$ or $\tau_D$, and $\mathrm{E}\left[\tau_D\right]\sim \left|\mathrm{E}\left[X_{\tau_D}\right]\right|$,\ $\mathrm{E}\left[\tau_D\right]\sim\mathrm{E}\left[\left|X_{\tau_D}-\mathrm{E}\left[X_{\tau_D}\right]\right|^2\right] $. Since $\mathrm{\Delta}^{\alpha/2,\lambda}_m$ is the infinitesimal generator of the anisotropic tempered stable process, we obtain the Feynman-Kac representation of the Dirichlet problem with the operator $\mathrm{\Delta}^{\alpha/2,\lambda}_m$. Therefore, averaging the generated trajectories of the stochastic process leads to the solution of the Dirichlet problem, which is also verified by numerical experiments.
\\
\\
\keywords{first exit problem; asymmetric tempered process; exponential decay; infinitesimal generator; Monte Carlo algorithm}


\end{abstract}

\section{Introduction}

L\'evy processes can effectively model the evolution processes with huge fluctuations, for example, fresh-water released by huge icebergs (Heinrich events), large fluctuations of the solar radiation steered by huge fluid outbursts on the surface of the sun. Sometimes, because of the particular  bounded physical space, the extremely large oscillation should be suppressed and then the tempered L\'evy processes are introduced \cite{43}. While describing the diffusion in complex inhomogeneous media, the nonisotropic tempered $\alpha$-stable processes are natural and reasonable choice. These and the related processes have been studied more or less from different aspects in recent years. For example, according to the characteristic functions of stochastic processes, the corresponding Fokker-Planck equations are derived \cite{1,2}; the numerical schemes are designed to solve the obtained Fokker-Planck equations \cite{3,4}; the relationship between mean square displacement (MSD) of stochastic process and time is discussed \cite{5}; and there are also many discussions on the applications of the stochastic processes and the corresponding macroscopic equations \cite{6,7,8,9}.

The first hitting time is defined as the time when a certain condition is fulfilled by the random variable of interest for the first time \cite{10}, which has a lot of potential applications. The example of first passage time naturally coming to our mind is the decision of an investor to buy or sell stock when its fluctuating prices reach a certain threshold \cite{14}. Here, we focus on the time and position distribution of first exit from a sphere for the nonisotropic tempered $\alpha$-stable processes, and use the results to numerically solve the corresponding Dirichlet problem.
The time and position distribution of first exit from a sphere are, respectively, defined as
 \begin{displaymath}
\tau_{B(0,r)}=\inf\left\{t>0:\ X_t\notin B(0,r)\right\},\qquad \rho(x)=P\left[X_{\tau_{B(0,r)}}\in \mathrm{d}x,\, \tau_{B(0,r)}<\infty\right],
 \end{displaymath}
where $X_t\in \mathrm{R}^d$ is a stochastic process, and $X_0=x,\ \ x\in\mathrm{R}^d$; the notation $B(0,r)$ is a sphere centred at the origin and the radius is $r$; the random variable $\tau_{B(0,r)}$ is the first exit time and $\rho(x)$ is the probability density function (PDF) of the first exit position. For the isotropic stochastic processes, there are some results on first exit time and position distribution obtained by establishing equations.
Letting $X_0=0$, when $X_t$ is the Brownian motion \cite{10,11,12,13},
\begin{displaymath}
\mathrm{E}\left[\tau_{B(0,r)}\right]=\frac{r^2}{2d},
\end{displaymath}
and $\rho(x)$ is uniform distribution on the boundary of $B(0,r)$, and the trajectories of the process hit $\partial B(0,r)$ in finite time with probability 1, because of the continuity and isotropy of Brownian motion. When $X_t$ is the $\alpha$-stable L\'evy process \cite{14,13,15,16,17,18,19,20,21,22,23,24,25,26,27},
\begin{displaymath}
\mathrm{E}\left[\tau_{B(0,r)}\right]=\frac{\Gamma(d/2)r^\alpha}{2^\alpha \Gamma(1+\alpha/2)\Gamma(d/2+\alpha/2)};
\end{displaymath}
due to the discontinuity of the paths of the processes, a particle starting at $X_0=0$, first escapes $B(0,r)$ and then lands in $B^c(0,r)$ (the complement of $B(0,r)$ in $\mathrm{R}^d$). Therefore, one needs to pay attention to the PDF $\rho(x)$ of the random variable $X_{\tau_{B(0,r)}}$ in $B^c(0,r)$. For $r<\infty$, $\rho(x)$ is given in \cite{28}.

Although there are many achievements for Brownian motion and $\alpha$-stable processes, little research has been done on the average of first exit time and the distribution of random variable $X_{\tau_{B(0,r)}}$, when the process is nonisotropic tempered process. Part of the reason is that it is difficult to get effective results by establishing equations. Tempered stable laws wipe off the probability of extremely large jumps, so that all moments of the tempered stable process exist. Thus, this can be preferable in application where the moments have a physical meaning. And the diffusion of particles may be nonisotropic due to environmental effects. In many practical applications, the nonisotropic tempered model may be more reasonable for simulating real data; so this paper concentrates on its Dirichlet problem and connection to first exit problems. The Dirichlet problem of Brownian motion is
 \begin{equation}\label{eq:1.1}
\begin{aligned}
\mathrm{\Delta}u(x)&=0,\qquad x\in D,\\
u(x)&=g_1(x),\qquad x\in\partial D,
\end{aligned}
\end{equation}
where $D$ is a domain in $\mathrm{R}^d$, $d\geq2$, with sufficiently smooth boundary, and $g_1$ is a continuous function on the boundary. The Dirichlet problem of $\alpha$-stable L\'evy process has the form
%
%
 \begin{equation}\label{eq:1.2}
\begin{aligned}
-(-\mathrm{\Delta})^{\alpha/2}u(x)&=0,\qquad x\in D,\\
u(x)&=g_2(x),\qquad x\in D^c,
\end{aligned}
\end{equation}
where $0<\alpha<2$, $g_2:\ D^c\to\mathrm{R}$ is a suitably regular function; and noting that $-(-\mathrm{\Delta})^{\alpha/2}$ is no longer a local operator, thus $\partial D$ is replaced by $D^c$ (the complement of $D$ in $\mathrm{R}^d$). Eq.\ \eqref{eq:1.1} is a very classical model, and it has been sufficiently studied in almost every aspect. As for Eq.\ \eqref{eq:1.2}, it attracts the wide interests of researchers in recent years, e.g., the discussion of the numerical schemes and their implementations \cite{29,30,31,32,33}; the main challenge of numerically solving the equation comes from the nonlocality of the fractional Laplacian and the weak singularity of the solution of \eqref{eq:1.2}. The well-known Feynman-Kac representation \cite{34,35} implies that if $u(x)$ is a solution to Eq.\ \eqref{eq:1.2}, then
\begin{equation}\label{eq:1.3}
u(x)=\mathrm{E}_x\left[g_2\left(X_{\tau_D}\right)\right],\qquad x\in D,
\end{equation}
where $\tau_D=\mathrm{inf}\left\{t>0:X_t\notin D\right\}$, and $X_t$ is the $\alpha$-stable L\'evy process; for Eq. \eqref{eq:1.1}, the similar representation holds, just replacing $g_2$ by $g_1$ in  Eq. \eqref{eq:1.3} and taking $X_t$ to be Brownian motion.
Eq.\ \eqref{eq:1.3} suggests that the solution of Dirichlet problem Eq.\ \eqref{eq:1.2} can be generated numerically by Monte Carlo algorithm \cite{36,37,38,39,40,41}. The advantage of Monte Carlo algorithm is that it can avoid the weak singularity and does not have the challenge of numerical cost for fractional Laplacian.



The Dirichlet problem for the asymmetric tempered fractional Laplacian, considered in this paper,  is \cite{42}
\begin{equation}\label{eq:1.4}
\begin{aligned}
\mathrm{\Delta}^{\alpha/2,\lambda}_mu(x)&=f(x),\qquad x\in D,\\
u(x)&=g(x),\qquad x\in D^c,
\end{aligned}
\end{equation}
where $\alpha\in(0,1)\cup(1,2)$, $f:\ D\to\mathrm{R}$ and $g:\ D^c\to\mathrm{R}$ are suitable functions; and
\begin{equation}\label{eq:1.5}
 \begin{aligned}
\mathrm{\Delta}^{\alpha/2,\lambda}_mu(x)=&c_{m,\alpha}\int_{\mathrm{R}^d\setminus\{0\}}\left[u(x-y)-u(x)
+\left(y\cdot\nabla_xu(x)\right)_{\chi_{\{|y|<1\}}}\right]\frac{m(y)}{\mathrm{e}^{\lambda |y|}|y|^{d+\alpha}}\mathrm{d}y
\end{aligned}
\end{equation}
with $m(x)$ denoting the probability distribution of particles spreading in direction and $c_{m,\alpha}$ being a normalized constant.
It seems that effectively solving Eq. \eqref{eq:1.4} is not an easy task because of the nonsymmetry and nonlocal property of Eq. \eqref{eq:1.5}. We demonstrate that the operator $\mathrm{\Delta}^{\alpha/2,\lambda}_m$ is an infinitesimal generator of the nonisotropic tempered stable process and present the Feynman-Kac representation of Eq.\ \eqref{eq:1.4}. Then, the Monte Carlo algorithm may be a feasible approach.

This paper is organized as follows. In the next section, we introduce the characteristic functions and compound Poisson forms of anisotropic tempered stable processes. In section 3, we estimate all the moments of $\left|X_{\tau_{B(0,r)}}\right|$ and $\tau_{B(0,r)}$; and the relationship between $X_{\tau_{B(0,r)}}$ and $\tau_{B(0,r)}$ is given in the mean sense. In section 4, we obtain the Feynman-Kac representation of the Dirichlet problem for the anisotropic tempered fractional Laplacian $\mathrm{\Delta}^{\alpha/2,\lambda}_m$. The numerical experiments are performed in section 5. Finally, we conclude the paper with some discussions in section 6.
\section{Tempered stable processes with L\'evy symbol and notations} \label{sec:2}
Let $X_t\,(t\geq0)$ be the isotropic tempered stable process in $\mathrm{R}^d$. Then its characteristic function $\hat{p}(k,t)=\mathrm{E}\left[\mathrm{e}^{i(k\cdot X_t)}\right]=\mathrm{e}^{t \psi(k)}$ \cite{43,44}, where the L\'evy symbol 
 \begin{equation}\label{eq:2.1}
 \psi(k)=\int_{\mathrm{R}^d\setminus\{0\}}\left(\mathrm{e}^{i(k\cdot x)}-1-i(k\cdot x)_{\chi\{|x|<1\}}\right)\nu(\mathrm{d}x)
 \end{equation}
 and the L\'evy measure
\begin{displaymath}
\nu(\mathrm{d}x)=c_\alpha\mathrm{e}^{-\lambda |x|}|x|^{-\alpha-d}
\end{displaymath}
with $c_\alpha=\frac{\Gamma(d/2)}{2\pi^{d/2}|\Gamma(-\alpha)|}$, $0<\alpha<2$, and $\alpha\neq1$. For the nonisotropic diffusion, the L\'evy measure is given as
\begin{displaymath}
\nu(\mathrm{d}x)=c_{m,\alpha}\frac{m(x)}{\mathrm{e}^{\lambda |x|}|x|^{d+\alpha}}\mathrm{d}x;
\end{displaymath}
to help understand the meaning of $m(x)$, one can notice that $\nu(\mathrm{d}x)$ has the polar coordinate form
\begin{displaymath}
c_{m,\alpha}\frac{m(\theta)}{\mathrm{e}^{\lambda r}r^{1+\alpha}}\mathrm{d}\theta\mathrm{d}r,
\end{displaymath}
where $r\ge0$, $\theta\in[0,2\pi)$ represents the direction, $m(\theta)$ is the probability distribution of particles in $\theta$-direction \cite{42}, $c_{m,\alpha}$ is the normalized constant;
and the anisotropic diffusion equation is
\begin{displaymath}
\frac{\partial p(x,t)}{\partial t}=\mathrm{\Delta}^{\alpha/2,\lambda}_mp(x,t).
\end{displaymath}
The definitions of the two special cases of the tempered fractional Laplacian are given as  \cite{42}  

$Case\ I$:  $0<\alpha<1$ or $m(y)$ is symmetric,
 \begin{equation}\label{eq:2.2}
\mathrm{\Delta}^{\alpha/2,\lambda}_mp(x,t)=c_{m,\alpha}\int_{\mathrm{R}^d\setminus\{0\}}\left[p(x-y,t)-p(x,t)\right]\frac{m(y)}{\mathrm{e}^{\lambda |y|}|y|^{d+\alpha}}\mathrm{d}y
\end{equation}

$Case\ II$:  $1<\alpha<2$ and $m(y)$ is asymmetric,
 \begin{equation}\label{eq:2.3}
 \begin{aligned}
\mathrm{\Delta}^{\alpha/2,\lambda}_mp(x,t)=&c_{m,\alpha}\int_{\mathrm{R}^d\setminus\{0\}}\left[p(x-y,t)-p(x,t)
+\left(y\cdot\nabla_xp(x,t)\right)_{\chi_{\{|y|<1\}}}\right]\frac{m(y)}{\mathrm{e}^{\lambda |y|}|y|^{d+\alpha}}\mathrm{d}y\\
&-c_{m,\alpha}\Gamma(1-\alpha)\lambda^{\alpha-1}\left(\mathbf{b}\cdot\nabla_xp(x,t)\right),
\end{aligned}
\end{equation}
where $\mathbf{b}=\int_{|\mathbf{\phi}|=1}\mathbf{\phi}m(\mathbf{\phi})\mathrm{d}\mathbf{\phi}$.
From Eq.\ \eqref{eq:2.2} and Eq.\ \eqref{eq:2.3}, we obtain the L\'evy symbols of the corresponding  anisotropic tempered stable processes,
\begin{equation}\label{eq:2.4}
 \psi(k)=c_{m,\alpha}\int_{R^d\setminus\{0\}}\left(\mathrm{e}^{i(k\cdot y)}-1\right)\frac{m(y)}{\mathrm{e}^{\lambda |y|}|y|^{d+\alpha}}\mathrm{d}y
 \end{equation}
and
\begin{equation}\label{eq:2.5}
\begin{aligned}
\psi(k)=&c_{m,\alpha}\int_{R^d\setminus\{0\}}\left(\mathrm{e}^{i(k\cdot y)}-1-i(k\cdot y)_{\chi_{\{|y|<1\}}}\right)\frac{m(y)}{\mathrm{e}^{\lambda |y|}|y|^{d+\alpha}}\mathrm{d}y\\
&+c_{m,\alpha}\Gamma(1-\alpha)\lambda^{\alpha-1}i\left(\mathbf{b}\cdot k\right),
\end{aligned}
\end{equation}
which indicates that the anisotropic tempered stable process $X_t\,(t\geq0)$ can be expressed by compound Poisson process. So, we have
 \begin{equation}\label{eq:2.6}
 X_t=\sum^{N(t)}_{i=0}Z_i,\qquad 0<\alpha<1
\end{equation}
and
 \begin{equation}\label{eq:2.7}
 X_t=\sum^{N(t)}_{i=0}Z_i-t\mathbf{\bar{b}},\qquad 1<\alpha<2,
\end{equation}
%
where
\begin{equation}
\begin{aligned}
\mathbf{\bar{b}}=&-c_{m,\alpha}\int_{R^d\setminus\{0\}} y_{\chi\{|y|<1\}}\frac{m(y)}{\mathrm{e}^{\lambda |y|}|y|^{d+\alpha}}\mathrm{d}y\\
&+c_{m,\alpha}\Gamma(1-\alpha)\lambda^{\alpha-1}\left(\mathbf{b}\right);
\end{aligned}
\end{equation}
 $Z_0$, $Z_1$, $Z_2$, $\cdots$, are a sequence of independent and identically distributed (i.i.d.) random variables taking valves in $\mathrm{R}^d$; the distribution of $Z_i$ is $\frac{m(x)\nu(\mathrm{d}x)}{\nu(|x|>0)}$; and $N(t)$ has a Poisson distribution
 \begin{equation*}
 P\left(N(t)=n\right)=\mathrm{e}^{-\mu t}\frac{(\mu t)^n}{n!},
 \end{equation*} %
where $\mu=c_{m,\alpha}\nu(|x|>0)$ is renewal intensity. Let $\eta_i$ be the waiting time between the ($i-1$)-th and $i$-th jumps. Then
\begin{equation*}
T[n]=\sum^{n}_{i=1}\eta_i,
\end{equation*}
which leads to $N(t)=\max\{n\ge0:\ T[n]\le t\}$.

Since the stable process is a compound Poisson process, naturally we can consider all the moments of $\tau_{B(0,r)}$ and $\left|X_{\tau_{B(0,r)}}\right|$ based on the compound Poisson processes.

\section{First exit position and time} \label{sec:3}
The average first exit time \cite{45} is a useful observation; here we provide the estimate of it for the processes discussed in the paper.
Define
\begin{displaymath}
M_{N(t)}=\max_{0\le s\le t}|X_s|.
\end{displaymath}

\begin{thm} \label{th:1}
Let $D$ be a bounded domain in $\mathrm{R}^d$ and $X_t$ the stochastic process $\mathrm{Eq}$.\ \eqref{eq:2.6} or $\mathrm{Eq}$.\ \eqref{eq:2.7}, $D\subseteq B(x,r)$. If $P(|Z|>2r)>0$, then
\begin{displaymath}
\mathrm{E}\left[\tau_D\right]\le\frac{1}{\mu P(|Z|>2r)}.
\end{displaymath}
\end{thm}
\begin{proof}
Since $D$ is a bounded domain, there exists a sphere $B(x,r),\ 0<r<\infty$ such that $D$ is its subset. Thus, we have
\begin{equation}\label{eq:3.1}
P\left(\tau_D\le\tau_{B(x,r)}\right)=1.
\end{equation}

For $0<\alpha<1$, computing probabilities by conditioning, we have
\begin{equation*}
P\left(M_{N(t)}\le r\right)=\sum^\infty_{n=0}P\left(M_{n}\le r\right)\mathrm{e}^{-\mu t}\frac{(\mu t)^n}{n!}.
\end{equation*}
Since
\begin{equation*}
P\left(M_{n}\le r\right)\le P\left(\max_{0\le i\le n}|Z_i|\le 2r\right)=\left(P(|Z|\le 2r)\right)^n,
\end{equation*}
there exists
\begin{equation}\label{eq:3.2}
P\left(M_{N(t)}\le r\right)\le \mathrm{e}^{-\mu tP(|Z|>2r)}.
\end{equation}

Noting that $P\left({\tau_{B(0,r)}>t}\right)=P\left(M_{N(t)}\le r\right)$, Eq.\ \eqref{eq:3.2} implies
\begin{equation*}
\begin{aligned}
\mathrm{E}\left[\tau_{B(0,r)}\right]&=\int^\infty_0P\left({\tau_{B(0,r)}>t}\right)\mathrm{d}t\\
&=\int^\infty_0P\left(M_{N(t)}\le r\right)\mathrm{d}t\\
&\le\int^\infty_0\mathrm{e}^{-\mu tP(|Z|>2r)}\mathrm{d}t\\
&=\frac{1}{\mu P(|Z|>2r)}.
\end{aligned}
\end{equation*}

Similarly, when $1<\alpha<2$, we also have
\begin{equation*}
P\left(M_{N(t)}\le r\right)\le P\left(\max_{0\le i\le N(t)}|Z_i|\le 2r\right)= \mathrm{e}^{-\mu tP(|Z|>2r)},
\end{equation*}
which means
\begin{equation*}
\mathrm{E}\left[\tau_{B(0,r)}\right]\le\frac{1}{\mu P(|Z|>2r)}.
\end{equation*}
Combining above and Eq.\ \eqref{eq:3.1} leads to
\begin{equation*}
\mathrm{E}\left[\tau_{D}\right]\le\mathrm{E}\left[\tau_{B(0,r)}\right]\le\frac{1}{\mu P(|Z|>2r)}.
\end{equation*}
\end{proof}

\begin{cor} \label{cor:2}
If $C<\mu P(|Z|>2r)$, we have
\begin{equation}\label{eq:3.3}
\mathrm{E}\left[\mathrm{e}^{C\tau_{B(0,r)}}\right]\le\frac{C}{\mu P(|Z|>2r)-C}+1.
\end{equation}
\end{cor}
\begin{proof}
The proof of Theorem 1 shows that
\begin{equation}\label{eq:3.4}
\begin{aligned}
\int^\infty_0\mathrm{e}^{Ct}P\left({\tau_{B(0,r)}>t}\right)\mathrm{d}t&=\int^\infty_0\mathrm{e}^{Ct}P\left(M_{N(t)}\le r\right)\\
&\le\int^\infty_0\mathrm{e}^{Ct}\mathrm{e}^{-\mu tP(|Z|>2r)}\mathrm{d}t\\
&=\frac{1}{\mu P(|Z|>2r)-C}.
\end{aligned}
\end{equation}
Let $f(t)$ be the PDF of $\tau_{B(0,r)}$. Making integration by parts leads to
\begin{equation*}
\begin{aligned}
\int^\infty_0\mathrm{e}^{Ct}P\left({\tau_{B(0,r)}>t}\right)\mathrm{d}t &=-\frac{1}{C}+\frac{1}{C}\int^\infty_0\mathrm{e}^{Ct}f(t)\mathrm{d}t
\\
&
=-\frac{1}{C}+\frac{1}{C}\mathrm{E}\left[\mathrm{e}^{C\tau_{B(0,r)}}\right]\le\frac{1}{\mu P(|Z|>2r)-C},
\end{aligned}
\end{equation*}
which results in
\begin{equation*}
\mathrm{E}\left[\mathrm{e}^{C\tau_{B(0,r)}}\right]\le\frac{C}{\mu P(|Z|>2r)-C}+1.
\end{equation*}
\end{proof}
\begin{remark}
If $D$ is a bounded domain and $P(|Z|>2r)>0$, then all the moments of the first exit time of compound Poisson processes are finite. The PDF of $\tau_{B(0,r)}$ decays exponentially when $t$ is large enough.
The power of exponential decay becomes smaller for bigger $\alpha$ when $\alpha>1$;
refer to Fig. \ref{fig:1} (for the details of simulation, see appendix).
\end{remark}


\captionsetup[figure]{name={Fig.}}
\begin{figure}[H]
\centering
\includegraphics[scale=0.45]{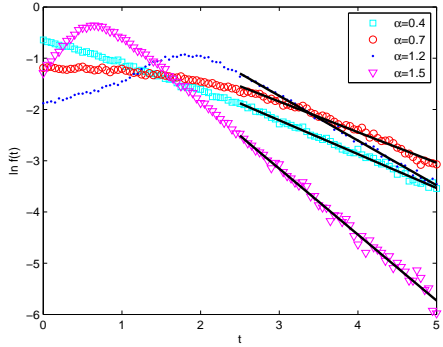}
\caption{PDF of $\tau_{B(0,1)}$ with $\lambda=0.01$ and $d=2$ for $\alpha=0.4$ (turquoise square), $\alpha=0.7$ (red circle), $\alpha=1.2$ (blue point), and $\alpha=1.5$ (magenta triangle). The parameters are: sample number $n=1.5\times10^5$, $X_0=[0,\ 0]$, $\Delta t=5\times10^{-4}$, $b=10$, $m(\theta)=1/{4\pi}$ for $\arg(\theta)\in(0,\ \pi)$ and $m(\theta)=3/{4\pi}$ for $\arg(\theta)\in(\pi,\ 2\pi)$.
}\label{fig:1}
\end{figure}

Another useful observation is the first exit position \cite{46}.
According to the compound Poisson processes, we estimate $\mathrm{E}\left[\left|X_{\tau_{B(0,r)}}\right|\right]$, where $\mathrm{E}$ denotes the average over space and time when the r.v. is $X_{\tau_{B(0,r)}}$.


\begin{prop} \label{prop:3}
Let $r>0$, $\xi>0$, $X_t=\sum^{N(t)}_{i=0}Z_i$ be compound Poisson process and $Z,\, Z_1, \, Z_2, \cdots$ be i.i.d.. Then
\begin{equation*}
P\left(|Z|>2r+\xi\right)G(\mu)\le P\left(\left|X_{\tau_{B(0,r)}}\right|-r>\xi\right)\le P\left(|Z|>\xi\right)G(\mu),
\end{equation*}
where $G(\mu)=\int^\infty_0\left(1-\mathrm{e}^{-\mu t}\right)f(t)\mathrm{d}t$, and $f(t)$ is the PDF of $\tau_{B(0,r)}$.
\end{prop}
\begin{proof}
$P\left(\left|X_{T_{B(0,r)}}\right|-r>\xi\right)$
\begin{equation*}
\begin{aligned}
&=\int^\infty_0P\left(\left|X_t\right|>r+\xi\right)f(t)\mathrm{d}t\\
&=\int^\infty_0\sum^\infty_{n=1}P\left(\left|\sum^n_{i=0}Z_i\right|>r+\xi\right)\mathrm{e}^{-\mu t}\frac{(\mu t)^n}{n!}f(t)\mathrm{d}t\\
&=\int^\infty_0\sum^\infty_{n=1}\int_{|x|\le r}P\left(\left|Z_n+x\right|>r+\xi\right)P\left(M_{n-1}\le r,\ \sum^{n-1}_{i=0}Z_i\in\mathrm{d}x\right)\mathrm{e}^{-\mu t}\frac{(\mu t)^n}{n!}f(t)\mathrm{d}t.
\end{aligned}
\end{equation*}
On the other hand, for $|x|\le r$, we have
\begin{equation}\label{eq:3.5}
P\left(\left|Z\right|>2r+\xi\right)\le P\left(\left|Z_n+x\right|>r+\xi\right)\le P\left(\left|Z\right|>\xi\right).
\end{equation}
Then
\begin{equation*}
P\left(\left|X_{T_{B(0,r)}}\right|-r>\xi\right)\le P\left(|Z|>\xi\right)\int^\infty_0\left(1-\mathrm{e}^{-\mu t}\right)f(t)\mathrm{d}t
\end{equation*}
and
\begin{equation*}
P\left(\left|X_{T_{B(0,r)}}\right|-r>\xi\right)\ge P\left(|Z|>2r+\xi\right)\int^\infty_0\left(1-\mathrm{e}^{-\mu t}\right)f(t)\mathrm{d}t.
\end{equation*}
The proof is completed.
\end{proof}

By Proposition \ref{prop:3}, for $\xi>1$, there exists
\begin{equation}\label{eq:3.6}
P\left(|Z|>r(1+\xi)\right)G(\mu)\le P\left(\frac{\left|X_{\tau_{B(0,r)}}\right|}{r}>\xi\right)\le P\left(|Z|>r(\xi-1)\right)G(\mu).
\end{equation}

While for $X_t=\sum^{N(t)}_{i=0}Z_i-t\mathbf{\bar{b}}$, the estimate of $P\left(\left|X_{\tau_{B(0,r)}}\right|-r>\xi\right)$ is slightly different. Because there are two ways to escape from the sphere $B(0,r)$, i.e., the escape is due to the shifted term $t\mathbf{\bar{b}}$ or the jump $Z_{N(t)}$. Let $J^c$ and $J$ represent two escape modes, respectively.

Define
\begin{equation*}
J^c=\{T\left[N(\tau_{B(0,r)})\right]<\tau_{B(0,r)}\}, \qquad  J=\{T\left[N(\tau_{B(0,r)})\right]=\tau_{B(0,r)}\}.
\end{equation*}
\begin{cor} \label{cor:4}
For $r>0$, $\xi> 0$, $X_t=\sum^{N(t)}_{i=0}Z_i-t\mathbf{\bar{b}}$, we have
\begin{equation*}
P\left(\left|X_{\tau_{B(0,r)}}\right|-r>\xi\right)\le P\left(|Z|>\xi\right)G(\mu).
\end{equation*}
\end{cor}
\begin{proof}
Because of the continuity of the shifted term $t\mathbf{\bar{b}}$, we have
\begin{equation*}
\begin{aligned}
P\left(\left|X_{\tau_{B(0,r)}}\right|-r>\xi\right)&=P\left(\left|X_{\tau_{B(0,r)}}\right|-r>\xi,J\right)+P\left(\left|X_{\tau_{B(0,r)}}\right|-r>\xi,J^c\right)\\
&=P\left(\left|X_{\tau_{B(0,r)}}\right|-r>\xi,J\right),
\end{aligned}
\end{equation*}
which implies that

$P\left(\left|X_{\tau_{B(0,r)}}\right|-r>\xi\right)$
\begin{equation*}
\begin{aligned}
&=\int^\infty_0\sum^\infty_{n=1}P\left(\left|\sum^n_{i=0}Z_i-t\mathbf{\bar{b}}\right|>r+\xi,J\right)\mathrm{e}^{-\mu t}\frac{(\mu t)^n}{n!}f(t)\mathrm{d}t\\
&=\int^\infty_0\sum^\infty_{n=1}\int_{|x|\le r}P\left(\left|Z_n+x\right|>r+\xi,J\right)P\left(\max_{0\le s<t}|X_s|\le r,\ \sum^{n-1}_{i=0}Z_i-t\mathbf{\bar{b}}\in\mathrm{d}x\right)\mathrm{e}^{-\mu t}\frac{(\mu t)^n}{n!}f(t)\mathrm{d}t.
\end{aligned}
\end{equation*}
For $|x|\le r$,
\begin{equation}\label{eq:3.7}
P\left(\left|Z_n+x\right|>r+\xi,J\right)\le P\left(\left|Z_n+x\right|>r+\xi\right)\le P\left(\left|Z\right|>\xi\right).
\end{equation}
Then, we have
\begin{equation*}
P\left(\left|X_{\tau_{B(0,r)}}\right|-r>\xi\right)\le P\left(|Z|>\xi\right)\int^\infty_0\left(1-\mathrm{e}^{-\mu t}\right)f(t)\mathrm{d}t.
\end{equation*}
The proof is completed.
\end{proof}

By Corollary \ref{cor:4}, for $\xi>1$, there exists
\begin{equation}\label{eq:3.8}
P\left(\frac{\left|X_{\tau_{B(0,r)}}\right|}{r}>\xi\right)\le P\left(|Z|>r(\xi-1)\right)G(\mu).
\end{equation}


These two results Eq.\ \eqref{eq:3.6} and Eq.\ \eqref{eq:3.8} lead to the estimates of the moments of $\left|X_{\tau_{B(0,r)}}\right|$.

\begin{thm} \label{th:5} 
Let $0<q<\infty$, $0<r<\infty$, and $X_t$ be the stochastic process  $\mathrm{Eq}$.\ \eqref{eq:2.6} or $\mathrm{Eq}$.\ \eqref{eq:2.7}. If $\mathrm{E}\left[|Z|^q\right]<\infty$, then
 \begin{displaymath}
\mathrm{E}\left[\left|X_{\tau_{B(0,r)}}\right|^q\right]<\infty.
 \end{displaymath}
\end{thm}
\begin{proof}
For $X_t=\sum^{N(t)}_{i=0}Z_i$, by Eq.\ \eqref{eq:3.6}, we have
\begin{equation*}
\begin{aligned}
\frac{\mathrm{E}\left[\left|X_{\tau_{B(0,r)}}\right|^q\right]}{r^q}&=\int^\infty_0q\xi^{q-1}P\left(\frac{\left|X_{\tau_{B(0,r)}}\right|}{r}>\xi\right)\mathrm{d}\xi\\
&\le2^q+q\int^\infty_2\xi^{q-1}P\left(\frac{\left|X_{\tau_{B(0,r)}}\right|}{r}>\xi\right)\mathrm{d}\xi\\
&\le2^q+G(\mu)q\int^\infty_2\xi^{q-1}P\left(\frac{|Z|}{r}>\frac{\xi}{2}\right)\mathrm{d}\xi\\
&=2^q+G(\mu)q(2/r)^q\int^\infty_r\xi^{q-1}P\left(|Z|>\xi\right)\mathrm{d}\xi.
\end{aligned}
\end{equation*}
The proof is completed.
\end{proof}
At the same time, the lower bound of the estimate can also be obtained as
\begin{equation*}
\begin{aligned}
\frac{\mathrm{E}\left[\left|X_{\tau_{B(0,r)}}\right|^q\right]}{r^q}&\ge\int^\infty_1q\xi^{q-1}P\left(\frac{\left|X_{\tau_{B(0,r)}}\right|}{r}>\xi\right)\mathrm{d}\xi\\
&\ge G(\mu)q\int^\infty_1\xi^{q-1}P\left(\frac{|Z|}{r}>2\xi\right)\mathrm{d}\xi\\
&=G(\mu)q(2r)^{-q}\int^\infty_{2r}\xi^{q-1}P\left(|Z|>\xi\right)\mathrm{d}\xi.
\end{aligned}
\end{equation*}

Similarly, for $X_t=\sum^{N(t)}_{i=0}Z_i-t\mathbf{\bar{b}}$, by Eq. \eqref{eq:3.8}, there exists
\begin{equation*}
\frac{\mathrm{E}\left[\left|X_{\tau_{B(0,r)}}\right|^q\right]}{r^q}\le2^q+G(\mu)q(2/r)^q\int^\infty_r\xi^{q-1}P\left(|Z|>\xi\right)\mathrm{d}\xi.
\end{equation*}
Theorem \ref{th:5} shows that $\mathrm{E}\left[\left|X_{\tau_{B(0,r)}}\right|^q\right]<\infty$ for the tempered stable process $X_t$.

\begin{cor} \label{cor:6}
For the tempered stable process $X_t$, if $\beta\le\lambda$, then
\begin{equation*}
\mathrm{E}\left[\exp\left(\beta\left|X_{\tau_{B(0,r)}}\right|\right)\right]<\infty.
\end{equation*}
\end{cor}
\begin{proof}
For $0<\alpha<1$, according to Proposition \ref{prop:3}, there exists
\begin{equation*}
\begin{aligned}
&\mathrm{E}\left[\exp\left(\beta\left|X_{\tau_{B(0,r)}}\right|\right)\right]\\
&=\mathrm{e}^{\beta r}+\int^\infty_r\beta\mathrm{e}^{\beta \xi}P\left(\left|X_{\tau_{B(0,r)}}\right|>\xi\right)\mathrm{d}\xi\\
&=\mathrm{e}^{\beta r}+\int^\infty_0\beta\mathrm{e}^{\beta (\xi_0+r)}P\left(\left|X_{\tau_{B(0,r)}}\right|-r>\xi_0\right)\mathrm{d}\xi_0\\
&\le \mathrm{e}^{\beta r}+\int^\infty_0\beta\mathrm{e}^{\beta (\xi_0+r)}P\left(\left|Z\right|>\xi_0\right)\mathrm{d}\xi_0\\
&=\frac{c_{m,\alpha}\mathrm{e}^{\beta r}}{{\nu(|Z|>0)}}\int^\infty_0\mathrm{e}^{(\beta-\lambda) \xi_0}\xi_0^{-1-\alpha}\mathrm{d}\xi_0.
\end{aligned}
\end{equation*}

As $1<\alpha<2$, Corollary \ref{cor:4} implies
\begin{equation*}
\begin{aligned}
&\mathrm{E}\left[\exp\left(\beta\left|X_{\tau_{B(0,r)}}\right|\right)\right]\\
&\le\frac{c_{m,\alpha}\mathrm{e}^{\beta r}}{{\nu(|Z|>0)}}\int^\infty_0\mathrm{e}^{(\beta-\lambda) \xi_0}\xi_0^{-1-\alpha}\mathrm{d}\xi_0.
\end{aligned}
\end{equation*}
\end{proof}

Corollary \ref{cor:6} also shows that the PDF $p(x)$ of $\left|X_{\tau_{B(0,r)}}\right|$ decays exponentially, as $\left|X_{\tau_{B(0,r)}}\right|$ becomes large, confirmed by \figurename \ \ref{fig:2}.

\captionsetup[figure]{name={Fig.}}
\begin{figure}[H]
\centering
\includegraphics[scale=0.45]{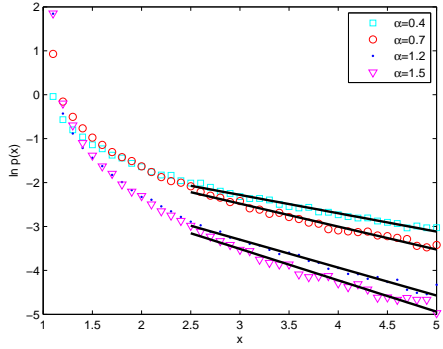}
\caption{Simulation of PDF of $\left|X_{\tau_{B(0,1)}}\right|$ with $\lambda=0.01$ and $d=2$ for  $\alpha=0.4$ (turquoise square), $\alpha=0.7$ (red circle), $\alpha=1.2$ (blue point), and $\alpha=1.5$ (magenta triangle). The other parameters are the same as those in \figurename \ \ref{fig:1}.}\label{fig:2}
\end{figure}

Generally, the moments of the tempered stable process $X_t$ are closely related to $t$. So is there a similar relationship between $\left|X_{\tau_D}\right|$ and $\tau_D$?
Firstly, let's discuss the symmetric tempered stable process; the following proposition demonstrates that the second moment of the process $|X_t|$ linearly increases with $t$.
\begin{prop}\label{lem-1}
If $X_t$ is the symmetric tempered stable process and $X_0=0$, then the second moment of it is independent of the dimension $d$, and there is
\begin{equation*}
\mathrm{E}\left[\left|X_t\right|^2\right]=\alpha|1-\alpha|\lambda^{\alpha-2}t.
\end{equation*}
\end{prop}

\begin{proof}
\begin{equation}\label{eq:3.9}
\begin{aligned}
\mathrm{E}\left[\left|X_t\right|^2\right]&=\sum^\infty_{n=0}\mathrm{E}\left[\left|\sum^{N(t)}_{i=0}Z_i\right|^2\bigg|N(t)=n\right] P\left[N(t)=n\right]\\
&=\sum^\infty_{n=0}\mathrm{E}\left[\left|\sum^n_{i=0}Z_i\right|^2\right] P\left[N(t)=n\right].
\end{aligned}
\end{equation}
The i.i.d. property and symmetry of $\nu(\mathrm{d}x)$ lead to
\begin{equation}\label{eq:3.10}
\begin{aligned}
\mathrm{E}\left[\left|X_t\right|^2\right]&=\mathrm{E}[|Z|^2]\sum^\infty_{n=0}n\mathrm{e}^{-\mu t}\frac{(\mu t)^n}{n!}\\
&=\frac{\alpha|1-\alpha|\lambda^{\alpha-2}}{\nu(|x|>0)}\mu t\\
&=\alpha|1-\alpha|\lambda^{\alpha-2}t,
\end{aligned}
\end{equation}
which completes the proof.
\end{proof}

For the symmetric tempered stable process, $\mathrm{E}\left[|X_t|^2\right]\sim t$. Can we also expect  $\mathrm{E}\left[\left|X_{\tau_{B(0,r)}}\right|^2\right]\sim \mathrm{E}\left[\tau_{B(0,r)}\right]$ \cite{47}? See the following theorem.

\begin{thm}\label{th:8}
Assume that the symmetric stochastic process $X_t$ has the stationary and independent increments, and
\begin{equation*}
\mathrm{E}\left[|X_t|^2\right]=Ct,
\end{equation*}
where $C$ is a constant; $D$ is a bounded domain, and $\mathrm{E}\left[\tau_D\right]<\infty$. Then
\begin{equation*}
\mathrm{E}\left[\left|X_{\tau_D}\right|^2\right]=C \cdot\mathrm{E}\left[\tau_D\right].
\end{equation*}
\end{thm}

\begin{proof}
Let $\tau(n_0)=\min\{\tau_D,n_0\}$, which is a finite stopping time. Since the increments are stationary and independent, there is
\begin{equation*}
\begin{aligned}
\mathrm{E}\left[\left|X_{\tau(n_0)}\right|^2\right]&=\lim_{n\to\infty}\mathrm{E}\left[\left|X\left(\frac{1}{n}\right)\mathbf{1}_{\left\{\tau(n_0)\ge\frac{1}{n}\right\}}
+X\left(\frac{2}{n}-\frac{1}{n}\right)\mathbf{1}_{\left\{ \tau(n_0)\ge\frac{2}{n}\right\}}+\cdots\right|^2\right]\\
&=\lim_{n\to\infty}\mathrm{E}\left[\left|X\left(\frac{1}{n}\right)\right|^2\right]\sum^{n_0n}_{i=1}P\left(\tau(n_0)\ge\frac{i}{n}\right)\\
&=C\lim_{n\to\infty}\frac{1}{n}\sum^{n_0n}_{i=1}P\left(\tau(n_0)\ge\frac{i}{n}\right)\\
&=C\int^{n_0}_0P\left(\tau(n_0)\ge t\right)\mathrm{d}t.
\end{aligned}
\end{equation*}
The fact
\begin{equation*}
\lim_{n_0\to\infty}\tau(n_0)=\tau_D,
\end{equation*}
leads to
\begin{equation}\label{eq:3.11}
\lim_{n_0\to\infty}\mathrm{E}\left[\left|X_{\tau(n_0)}\right|^2\right]=\mathrm{E}\left[\left|X_{\tau_D}\right|^2\right]
=C\mathrm{E}\left[\tau_D\right],
\end{equation}
which completes the proof.
\end{proof}

According to Theorem \ref{th:8}, one can get $\mathrm{E}\left[\tau_{B(0,r)}\right]$ of the symmetric Brownian motion. Since
\begin{equation*}
\mathrm{E}\left[\left|X_t\right|^2\right]=2dt,
\end{equation*}
which leads to
\begin{equation*}
\mathrm{E}\left[\tau_{B(0,r)}\right]=\mathrm{E}\left[\left|X_{\tau_{B(0,r)}}\right|^2\right]/{2d}=r^2/{2d}.
\end{equation*}
There's another way to prove Theorem \ref{th:8}. Let
\begin{equation*}
M_t=\left|X_t\right|^2-Ct.
\end{equation*}
Define the natural filtration of the process $X_t$ as
\begin{displaymath}
 \mathcal{G}^X_t=\sigma\left\{X_s: \ 0\leq s\leq t\right\}.
 \end{displaymath}
 Then, we have
 \begin{equation*}
 \begin{aligned}
\mathrm{E}\left[M_t|\mathcal{G}^X_s\right]&=\mathrm{E}\left[\left|X_t-X_s\right|^2|\mathcal{G}^X_s\right]-Ct+\left|X_s\right|^2 \qquad (s\le t)\\
&=\left|X_s\right|^2-Cs\\
&=M_s,
\end{aligned}
\end{equation*}
showing that $M_t$ is a martingale. Thus, using Doob's optional stopping theorem leads to
\begin{equation*}
\mathrm{E}\left[M_{\tau_D}\right]=\mathrm{E}\left[M_0\right]=0,
\end{equation*}
which implies that Theorem \ref{th:8} holds.

For the anisotropic tempered stable process $X_t$, we calculate its second moment. 
In the two dimensional case, define
\begin{equation*}
Z_i=(r_i\cos\theta_i,r_i\sin\theta_i),
\end{equation*}
where the PDF of $r_i$  is
\begin{equation} \label{PDFrad:r}
\frac{\mathrm{e}^{-\lambda r}r^{-1-\alpha}}{\int^\infty_0\mathrm{e}^{-\lambda r}r^{-1-\alpha}\mathrm{d}r},
\end{equation}
and the PDF of $\theta_i$ is $m(\theta)$, defined on $[0,2\pi]$. Note that $r_i$($\theta_i$) is i.i.d. random variable, and $r_i$ and $\theta_i$ are independent of each other.
When $0<\alpha<1$ in (\ref{PDFrad:r}), for Eq.\ \eqref{eq:2.6}, there exists
\begin{equation}\label{eq:3.12}
\begin{aligned}
\mathrm{E}\left[|X_t|^2\right]&=\sum^\infty_{n=0}\mathrm{E}\left[\left|\sum^n_{i=0}Z_i\right|^2\right]P\left[N(t)=n\right]\\
&=\sum^\infty_{n=0}\mathrm{E}\left[\left(\sum^n_{i=0}r_i\cos\theta_i\right)^2+\left(\sum^n_{i=0}r_i\sin\theta_i\right)^2\right]P\left[N(t)=n\right]\\
&=\sum^\infty_{n=0}\left(n\mathrm{E}\left[r^2_1\right]+n(n-1)\left(\mathrm{E}\left[r_1\right]\right)^2\left[\left(\mathrm{E}\left[\cos\theta_1\right]\right)^2
+\left(\mathrm{E}\left[\sin\theta_1\right]\right)^2\right]\right)P\left[N(t)=n\right]\\
&=\mathrm{E}\left[r^2_1\right]\mu t+\left(\mathrm{E}\left[r_1\cos\theta_1\right]\right)^2(\mu t)^2
+\left(\mathrm{E}\left[r_1\sin\theta_1\right]\right)^2(\mu t)^2.
\end{aligned}
\end{equation}
The MSD of $X_t$ is a linear function of $t$:
\begin{equation}\label{eq:3.13}
\begin{aligned}
&\mathrm{E}\left[\left|X_t-\mathrm{E}[X_t]\right|^2\right]\\
&=\sum^\infty_{n=0}\mathrm{E}\left[\left|\sum^n_{i=0}Z_i-\mathrm{E}\left[\sum^n_{i=0}Z_i\right]\right|^2\right]P\left[N(t)=n\right]\\
&=\sum^\infty_{n=0}\mathrm{E}\left[\left(\sum^n_{i=0}\left(r_i\cos\theta_i-\mathrm{E}[r_i\cos\theta_i]\right)\right)^2+\left(\sum^n_{i=0}\left(r_i\sin\theta_i-\mathrm{E}[r_i\sin\theta_i]\right)\right)^2\right]P\left[N(t)=n\right]\\
&=\sum^\infty_{n=0}\left(n\mathrm{E}\left[r^2_1\right]-n\left(\mathrm{E}\left[r_1\right]\right)^2\left[\left(\mathrm{E}\left[\cos\theta_1\right]\right)^2
+\left(\mathrm{E}\left[\sin\theta_1\right]\right)^2\right]\right)P\left[N(t)=n\right]\\
&=\mathrm{E}\left[r^2_1\right]\mu t-\left[\left(\mathrm{E}\left[r_1\cos\theta_1\right]\right)^2
+\left(\mathrm{E}\left[r_1\sin\theta_1\right]\right)^2\right] \mu t.
\end{aligned}
\end{equation}

For the three dimensional case, i.e., $d=3$, $Z_i=(r_i\sin\theta_i\cos\varphi_i,r_i\sin\theta_i\sin\varphi_i,r_i\cos\theta_i)$, and the probability distribution of the radial direction of $X_t$ is $m(\theta,\varphi)$, defined on the domain $[0,\pi]\times[0,2\pi]$. Using the same above steps leads to
\begin{equation}\label{eq:3.14}
\begin{aligned}
&\mathrm{E}\left[|X_t|^2\right]\\
&=\mathrm{E}\left[r^2_1\right]\mu t+\left(\left(\mathrm{E}\left[r_1\cos\theta_1\right]\right)^2
+\left(\mathrm{E}\left[r_1\sin\theta_1\sin\varphi_1\right]\right)^2
+\left(\mathrm{E}\left[r_1\sin\theta_1\cos\varphi_1\right]\right)^2\right)(\mu t)^2
\end{aligned}
\end{equation}
and
\begin{equation}\label{eq:3.15}
\begin{aligned}
&\mathrm{E}\left[\left|X_t-\mathrm{E}[X_t]\right|^2\right]\\
&=\mathrm{E}\left[r^2_1\right]\mu t-\left[\left(\mathrm{E}\left[r_1\cos\theta_1\right]\right)^2
+\left(\mathrm{E}\left[r_1\sin\theta_1\sin\varphi_1\right]\right)^2+\left(\mathrm{E}\left[r_1\sin\theta_1\cos\varphi_1\right]\right)^2\right] \mu t.
\end{aligned}
\end{equation}
When $1<\alpha<2$ in (\ref{PDFrad:r}), for the two and three dimensional cases, $\mathbf{\bar{b}}$, respectively, has the form $(\bar{b}_1,\bar{b}_2)$ and $(\bar{b}_1,\bar{b}_2,\bar{b}_3)$, and we have
\begin{equation*}
\begin{aligned}
\mathrm{E}\left[|X_t|^2\right]&=\sum^\infty_{n=0}\mathrm{E}\left[\left|\sum^n_{i=0}Z_i-t\mathbf{\bar{b}}\right|^2\right]P\left[N(t)=n\right]\\
&=\sum^\infty_{n=0}\mathrm{E}\left[\left(\sum^n_{i=0}r_i\cos\theta_i-t\bar{b}_1\right)^2+\left(\sum^n_{i=0}r_i\sin\theta_i-t\bar{b}_2\right)^2\right]P\left[N(t)=n\right]\\
&=\mathrm{E}\left[r^2_1\right]\mu t+\left(\mathrm{E}\left[r_1\cos\theta_1\right]\right)^2(\mu t)^2
+\left(\mathrm{E}\left[r_1\sin\theta_1\right]\right)^2(\mu t)^2\\
&-2\mu t^2\left(\mathrm{E}\left[r_1(\bar{b}_1\cos\theta_1+\bar{b}_2\sin\theta_1)\right]\right)+(|\mathbf{\bar{b}}|t)^2
\end{aligned}
\end{equation*}
and
\begin{equation*}
\begin{aligned}
&\mathrm{E}\left[|X_t|^2\right]\\
&=\mathrm{E}\left[r^2_1\right]\mu t+\left(\left(\mathrm{E}\left[r_1\cos\theta_1\right]\right)^2
+\left(\mathrm{E}\left[r_1\sin\theta_1\sin\varphi_1\right]\right)^2
+\left(\mathrm{E}\left[r_1\sin\theta_1\cos\varphi_1\right]\right)^2\right)(\mu t)^2\\
&-2\mu t^2\left(\mathrm{E}\left[r_1(\bar{b}_1\cos\theta_1+\bar{b}_2\sin\theta_1\sin\varphi_1+\bar{b}_3\sin\theta_1\cos\varphi_1)\right]\right)+(|\mathbf{\bar{b}}|t)^2.
\end{aligned}
\end{equation*}
The MSDs of $|X_t|$ with $1<\alpha<2$ in the two and three dimensional cases are, respectively, the same as the ones of $|X_t|$ with $0<\alpha<1$, i.e., Eq.\ \eqref{eq:3.13} and Eq.\ \eqref{eq:3.15}.
When $d\ge4$, one can similarly get the second moment and MSD of $X_t$ as
\begin{equation*}
\mathrm{E}\left[|X_t|^2\right]=\mathrm{E}\left[r^2_1\right]\mu t+C_1t^2
\end{equation*}
and
\begin{equation*}
\mathrm{E}\left[\left|X_t-\mathrm{E}[X_t]\right|^2\right]=C_2t,
\end{equation*}
where $C_1$ and $C_2$ vary with $\alpha,\ \lambda$, and the dimension.




The following theorem answers the relationship between the mean of first exit time $\tau_D$ and the moment/MSD of $X_{\tau_D}$.

\begin{thm}\label{th:9}
If the anisotropic stochastic process $X_t$ has the stationary and independent increments, and
\begin{equation*}
\mathrm{E}\left[X_t\right]=\mathbf{c}t
\end{equation*}
or
\begin{equation*}
\mathrm{E}\left[\left|X_t-\mathrm{E}[X_t]\right|^2\right]=Ct,
\end{equation*}
where $\mathbf{c}$ is a vector, and the constant $C$ depends on the dimension $d$. When $D$ is a bounded domain and $\mathrm{E}\left[\tau_D\right]<\infty$. Then
\begin{equation*}
\mathrm{E}\left[\tau_D\right]=\frac{1}{|\mathbf{c}|}\left|\mathrm{E}\left[X_{\tau_D}\right]\right|\\
\end{equation*}
or
\begin{equation*}
\mathrm{E}\left[\tau_D\right]=\frac{1}{C}\mathrm{E}\left[\left|X_{\tau_D}-\mathrm{E}\left[X_{\tau_D}\right]\right|^2\right].
\end{equation*}
\end{thm}

\begin{proof}
Let
\begin{equation*}
M_t=\left|X_t-\mathrm{E}[X_t]\right|^2-Ct.
\end{equation*}
Because of the stationary and independence of the increments, there exists
 \begin{equation}\label{eq:3.16}
 \begin{aligned}
\mathrm{E}\left[M_t|\mathcal{G}^X_s\right]&=\mathrm{E}\left[\left|X_{t-s}-\mathrm{E}[X_{t-s}]+X_s-\mathrm{E}[X_s]\right|^2|\mathcal{G}^X_s\right]-Ct \qquad (s\le t)\\
&=\mathrm{E}\left[\left|X_{t-s}-\mathrm{E}[X_{t-s}]\right|^2\right]+\left|X_s-\mathrm{E}[X_s]\right|^2-Ct\\
&=M_s,
\end{aligned}
\end{equation}
which implies that $M_t$ is a martingale.
Thus, by Doob's optional stopping theorem, we have
\begin{equation*}
\begin{aligned}
\mathrm{E}\left[M_{\tau_D}\right]&=\mathrm{E}\left[\left|X_{\tau_D}-\mathrm{E}\left[X_{\tau_D}\right]\right|^2\right]-C\mathrm{E}\left[\tau_D\right]\\
&=\mathrm{E}\left[M_0\right]=0.
\end{aligned}
\end{equation*}
Following the same analysis as above, we have
\begin{equation*}
\mathrm{E}\left[X_{\tau_D}\right]=\mathbf{c}\mathrm{E}\left[\tau_D\right].
\end{equation*}
The proof is completed.
\end{proof}
Theorem \ref{th:8} and Theorem \ref{th:9} show the relationship between first exit position and time for the anisotropic tempered stable process. Note that the method of proof of Theorem \ref{th:8} also applies for Theorem \ref{th:9}.

\section{Exact solution of Dirichlet problem for the tempered fractional Laplacian} \label{sec:4}

Based on the results given in Section \ref{sec:3}, we provide the Feynman-Kac representation of Eq. \eqref{eq:1.4} with suitable functions $g$ and $f$.
The characteristic function of anisotropic tempered stable process with $\alpha\in(0,1)\cup(1,2)$ can be rewritten as \cite{42}
\begin{equation*}
\hat{p}(k,t)=\mathrm{E}\left[\mathrm{e}^{-i (k\cdot X_t)}\right]=\mathrm{exp}\left[t \cdot (-1)^{\lceil\alpha\rceil}\int_{|\mathbf{\phi}|=1}\left((\lambda+ik\cdot \mathbf{\phi})^\alpha-\lambda^\alpha\right)m(\mathbf{\phi})\mathrm{d}\mathbf{\phi}\right],
\end{equation*}
which satisfies
\begin{equation}\label{eq:4.2}
\frac{\mathrm{d}\hat{p}(k,t)}{\mathrm{d}t}= (-1)^{\lceil\alpha\rceil}\left[\int_{|\mathbf{\phi}|=1}\left((\lambda+ik\cdot \mathbf{\phi})^\alpha-\lambda^\alpha\right)m(\mathbf{\phi})\mathrm{d}\mathbf{\phi}\right]\hat{p}(k,t).
\end{equation}
Performing the inverse Fourier transform on (\ref{eq:4.2}) leads to
the Fokker-Planck equation
\begin{equation}\label{eq:4.1}
\frac{\mathrm{d}p(X_t,t)}{\mathrm{d}t}=\mathrm{\Delta}^{\alpha/2,\lambda}_mp(X_t,t),
\end{equation}
where the operator $\mathrm{\Delta}^{\alpha/2,\lambda}_m$ is defined in (\ref{eq:1.5}).

The linear operator semigroup ($\mathrm{T}_t,\ t>0$) of the stochastic process $X_t$ is defined by
\begin{equation*}
\begin{aligned}
\mathrm{T}_tu(x)&=\mathrm{E}\left[u\left(X_t\right)\big|X_0=x\right],\\
\mathrm{T}_0u(x)&=u(x),\\
\mathrm{T}_t\cdot\mathrm{T}_s&=\mathrm{T}_{t+s}.
\end{aligned}
\end{equation*}
For $\mathrm{T}_t$, we have
\begin{equation}\label{eq:4.3}
\mathrm{A}u(x)=\lim_{t\to 0}\frac{\mathrm{T}_tu(x)-u(x)}{t},
\end{equation}
where $\mathrm{A}$ is the infinitesimal generator of $X_t$.

\begin{prop}\label{lemma2.1}
For the nonisotropic tempered $\alpha$-stable ($\alpha\in(0,1)\cup(1,2)$) process $X_t$, the operator $\mathrm{\Delta}^{\alpha/2,\lambda}_m$ is its infinitesimal generator.
\end{prop}
\begin{proof}
The Fourier transform (FT) of $\mathrm{T}_tu(x)$ is
\begin{equation}\label{eq:4.4}
\begin{aligned}
\mathscr{F}\left[\mathrm{T}_tu(x)\right]&=\mathscr{F}\left[\int^{+\infty}_{-\infty}u\left(X_t+x\right)p\left(X_t,t\right)\mathrm{d}X_t\right]\\
&=\int^{+\infty}_{-\infty}\mathscr{F}\left[u\left(X_t+x\right)\right]p\left(X_t,t\right)\mathrm{d}X_t\\
&=\int^{+\infty}_{-\infty}\hat{u}(k)\mathrm{e}^{-i(k\cdot X_t)}p\left(X_t,t\right)\mathrm{d}X_t\\
&=\hat{u}(k)\hat{p}(k,t),
\end{aligned}
\end{equation}
where $\hat{u}(k)$ is the FT of $u(x)$, and the Fubini Theorem is used in the second equality.
Combining Eq.\ \eqref{eq:4.3} and Eq.\ \eqref{eq:4.4}, we have
\begin{equation}\label{eq:4.5}
\begin{aligned}
&\lim_{t\to0}\frac{\mathscr{F}\left[\mathrm{T}_tu(x)\right]-\mathscr{F}\left[u(x)\right]}{t}\\
=&\lim_{t\to0}\frac{\hat{u}(k)\hat{p}(k,t)-\hat{u}(k)}{t}\\
=&\lim_{t\to0}\frac{\hat{u}(k)\left[\hat{p}(k,t)-1\right]}{t}\\
=&(-1)^{\lceil\alpha\rceil}\left[\int_{|\mathbf{\phi}|=1}\left((\lambda+ik\cdot \mathbf{\phi})^\alpha-\lambda^\alpha\right)m(\mathbf{\phi})\mathrm{d}\mathbf{\phi}\right]\hat{u}(k).
\end{aligned}
\end{equation}
Making the inverse FT on Eq. \eqref{eq:4.5}, from Eq. \eqref{eq:4.1} and Eq. \eqref{eq:4.2}, we have
\begin{equation*}
\begin{aligned}
&\lim_{t\to 0}\frac{\mathrm{T}_tu(x)-u(x)}{t}\\
=&\mathrm{\Delta}^{\alpha/2,\lambda}_mu(x),
\end{aligned}
\end{equation*}
which leads to
\begin{equation*}
\mathrm{A}u(x)=\mathrm{\Delta}^{\alpha/2,\lambda}_mu(x).
\end{equation*}
The proof is completed.
\end{proof}

The measurable real-valued function $g$ on a Borel set $\mathfrak{B}\left(\mathrm{R}^d\right)$ belongs to $\mathcal{L}_{\lambda}$ if it satisfies
\begin{equation*}
 |g(x)|\le C\exp(\lambda |x|),\quad |x|\ge l,
\end{equation*}
where $l$ and $C$ are bounded constants.
\begin{thm} \label{th:11}
Suppose that $D$ is a bounded domain in $\mathrm{R}^d\,(d\geq2)$, $g$ is a uniformly continuous function on $D^c$, and $g(x)\in\mathcal{L}_{\lambda}\left(D^c\right)$. Moreover, assume that $f$ is a continuous bounded function in the domain $\overline{D}$. Then there exists an unique continuous solution to $\mathrm{Eq}.\ \eqref{eq:1.4}$:
\begin{displaymath}
 u(x)=\mathrm{E}_{x}\left[g\left(X_{\tau_D}\right)\right]-\mathrm{E}_{x}\left[\int^{\tau_D}_0f\left(X_s\right)\mathrm{d}s\right].
 \end{displaymath}
\end{thm}

To prove Theorem \ref{th:11}, $\mathrm{E}_{x}\left[g\left(X_{\tau_D}\right)\right]$ and $\mathrm{E}_{x}\left[\int^{\tau_D}_0f\left(X_s\right)\mathrm{d}s\right]$ must firstly exist. Since $D$ is a bounded domain, one can find a sphere $B(0,l),\ 0<l<\infty$, such that $D$ is a subset of $B(0,l)$. Calculating expectations by conditioning, we have
\begin{equation}\label{eq:4.6}
\begin{aligned}
\mathrm{E}_{x}\left[g\left(X_{\tau_D}\right)\right]&=\mathrm{E}_{x}\left[g\left(X_{\tau_D}\right)|\chi\{\tau_D=\tau_{B(0,l)}\}\right]
+\mathrm{E}_{x}\left[g\left(X_{\tau_D}\right)|\chi\{\tau_D<\tau_{B(0,l)}\}\right]\\
&\le\mathrm{E}_{x}\left[g\left(X_{\tau_{B(0,l)}}\right)\right]+\max_{x\in B(0,l)\setminus D}g(x).
\end{aligned}
\end{equation}
The uniform continuity of $g(x)$ leads to
\begin{equation*}
\max_{x\in B(0,l)\setminus D}g(x)<\infty.
\end{equation*}
By Corollary \ref{cor:6}, there exists
\begin{equation*}
\begin{aligned}
\mathrm{E}_x\left[g\left(X_{\tau_{B(0,l)}}\right)\right]&\le \mathrm{E}_x\left[\left|g\left(X_{\tau_{B(0,l)}}\right)\right|\right]\\
&\le C\mathrm{E}_x\left[\exp\left(\lambda \left|X_{\tau_{B(0,l)}}\right|\right)\right]\\
&<\infty.
\end{aligned}
\end{equation*}
So, finally we arrive at $\mathrm{E}_{x}\left[g\left(X_{\tau_D}\right)\right]<\infty$. According to Theorem \ref{th:1} and $f<\infty$, there is
\begin{displaymath}
\begin{aligned}
\mathrm{E}_x\left[\int^{\tau_D}_0f\left(X_s\right)\mathrm{d}s\right]&\le \sup_{0<t<\tau_D}f\left(X_t\right)\mathrm{E}_x\left[\int^{\tau_D}_0\mathrm{d}s\right]\\
&= \sup_{0<t<\tau_D}f\left(X_t\right)\mathrm{E}_x\left[\tau_D\right]\\
&<\infty.
\end{aligned}
\end{displaymath}

Proof of Theorem \ref{th:11}. Let
\begin{displaymath}
 M_t=u(X_t)-u(x)-\int^t_0\mathrm{A}u(X_s)\mathrm{d}s.
 \end{displaymath}
 Then
 \begin{equation}\label{eq:4.7}
 \begin{aligned}
 \mathrm{E}_x\left[M_{t+h}|\mathcal{G}^X_t\right]&=
 \mathrm{E}_x\left[u(X_{t+h})|\mathcal{G}^X_t\right]-u(x)-\mathrm{E}_x\left[\int^{t+h}_0\mathrm{A}u(X_s)\mathrm{d}s|\mathcal{G}^X_{t}\right]\\
 &=\mathrm{E}_x\left[u(X_{t+h})|X_t\right]-u(x)-\int^{t+h}_0\mathrm{E}_x\left[\mathrm{A}u(X_s)\mathrm{d}s|X_{t}\right]\\
 &=\mathrm{E}_x\left[u(X_{t+h}-X_t+X_t)|X_t\right]-u(x)-\int^{t+h}_0\mathrm{E}_x\left[\mathrm{A}u(X_s)\mathrm{d}s|X_{t}\right]\\
 &=\mathrm{E}_x\left[u(X_{h}+X_t)|X_t\right]-u(x)-\int^{t+h}_t\mathrm{E}_x\left[\mathrm{A}u(X_s)\mathrm{d}s|X_{t}\right]-\int^{t}_0\left[\mathrm{A}u(X_s)\mathrm{d}s\right]\\
 &=\mathrm{T}_hu(X_t)-u(x)-\int^{t+h}_t\mathrm{T}_{s-t}\left[\mathrm{A}u(X_t)\right]\mathrm{d}s-\int^{t}_0\left[\mathrm{A}u(X_s)\mathrm{d}s\right],\\
 \end{aligned}
 \end{equation}
 where the Fubini Theorem is used in the second equality, and the stationarity and independence of the increments of the process $X_t$ are used in the fourth equality.


Combining $\mathrm{T}_s\cdot\mathrm{T}_t=\mathrm{T}_{t+s}$ and Eq.\ \eqref{eq:4.3}, there is
\begin{equation*}
\begin{aligned}
\mathrm{A}\mathrm{T}_su(x)&=\lim_{t\to 0}\frac{\mathrm{T}_t\left[\mathrm{T}_su(x)\right]-\mathrm{T}_su(x)}{t}\\
&=\lim_{t\to 0}\frac{\mathrm{T}_s\left[\mathrm{T}_tu(x)-u(x)\right]}{t}\\
&=\mathrm{T}_s\mathrm{A}u(x),
\end{aligned}
\end{equation*}
which leads to
\begin{equation}\label{eq:4.8}
\begin{aligned}
\mathrm{E}_x\left[M_{t+h}|\mathcal{G}^X_t\right]&=
\mathrm{T}_hu(X_t)-u(x)-\int^{t+h}_t\mathrm{A}\mathrm{T}_{s-t}\left[u(X_t)\right]\mathrm{d}s-\int^{t}_0\left[\mathrm{A}u(X_s)\mathrm{d}s\right]\\
&=\mathrm{T}_hu(X_t)-u(x)-\mathrm{T}_hu(X_t)+\mathrm{T}_0u(X_t)-\int^{t}_0\left[\mathrm{A}u(X_s)\mathrm{d}s\right]\\
&=u(X_t)-u(x)-\int^{t}_0\left[\mathrm{A}u(X_s)\mathrm{d}s\right]\\
&=M_t.
\end{aligned}
 \end{equation}
Let $\Omega$ denote the set of outcomes of the random experiment $X_t\,(t\geq0)$ with fixed $t$, and $\mathcal{D}=\{\emptyset,\Omega\}$.
By the double expectation formula, we have
From  Eq.\ \eqref{eq:4.8} and the double expectation formula, we have
\begin{equation}\label{eq:4.10}
\begin{aligned}
\mathrm{E}_x\left[\mathrm{E}_x\left[M_{\tau_D}|\mathcal{G}^X_0\right]|\mathcal{D}\right]
&=\mathrm{E}_x\left[M_{\tau_D}|\mathcal{D}\right]\\
&=\mathrm{E}_x\left[M_0|\mathcal{D}\right]\\
&=\mathrm{E}_x\left[M_0\right].
\end{aligned}
\end{equation}
Combining Eq.\ \eqref{eq:4.7} and Eq.\ \eqref{eq:4.10} leads to
\begin{equation}\label{eq:4.11}
\mathrm{E}_x\left[u(X_{\tau_D})\right]-u(x)-\mathrm{E}_x\left[\int^{\tau_D}_0\mathrm{A}u(X_s)\mathrm{d}s\right]=\mathrm{E}_x\left[M_0\right],
\end{equation}
which results in
\begin{equation*}
u(x)=\mathrm{E}_x\left[g(X_{\tau_D})\right]-\mathrm{E}_x\left[\int^{\tau_D}_0f(X_s)\mathrm{d}s\right].
\end{equation*}
The proof is completed. 

Theorem \ref{th:11} shows that the solution of Eq.\ \eqref{eq:1.4} can be obtained numerically by straightforward Monte Carlo simulations of the path of $X_t$ until first exit from $D$.
By the strong law of large numbers, we have
\begin{equation}\label{eq:4.12}
\begin{aligned}
&\lim_{n\to\infty}\frac{1}{n}\sum^n_{i=1}\left[g\left(X^i_{\tau^i_D}\right)-\int_0^{\tau^i_D}f\left(X^i_s\right)\mathrm{d}s\right]\\
&=\mathrm{E}_x\left[g(X_{\tau_D})\right]-\mathrm{E}_x\left[\int^{\tau_D}_0f(X_s)\mathrm{d}s\right]=u(x),\qquad \mathrm{almost} \ \mathrm{surely},
\end{aligned}
\end{equation}
where $X^i_{\tau^i_D}$ are i.i.d. copies of $X_{\tau_D}$ starting from $x\in D$.
Practically, it is impossible to take the limit in Eq.\ \eqref{eq:4.12}, so one needs to truncate the series of estimate by taking sufficiently large $n$. Then, there is a truncation error
\begin{equation}\label{Numer_error}
\mathrm{error}=\frac{1}{n}\sum^n_{i=1}\left[g\left(X^i_{\tau^i_D}\right)-\int_0^{\tau^i_D}f\left(X^i_s\right)\mathrm{d}s\right]-u(x).
\end{equation}

According to (\ref{eq:4.6}), if $g^2(x)\in\mathcal{L}_{\lambda}\left(D^c\right)$, then $\mathrm{E}\left[\left(g(X_{\tau_D}\right)^2\right]<\infty$. From Corollary \ref{cor:2}, we have 
 \begin{equation}\label{eq:4.13}
 \mathrm{E}_x\left[\left(\int^{\tau_D}_0f(X_s)\mathrm{d}s\right)^2\right]\le\left(\sup_{0<t<\tau_D}f\left(X_t\right)\right)^2\mathrm{E}_x\left[\left(\int^{\tau_D}_0\mathrm{d}s\right)^2\right]
 <\infty.
\end{equation}
Then, there exists 
\begin{equation}\label{eq:4.14}
\mathrm{E}_x\left[\left(g(X_{\tau_D})-\int^{\tau_D}_0f(X_s)\mathrm{d}s\right)^2\right]<\infty.
\end{equation}
Using the central limit theorem, in the sense of weak convergence, we have
\begin{equation}\label{eq:4.15}
\begin{aligned}
&\lim_{n\to\infty} n^{1/2}\left(\frac{1}{n}\sum^n_{i=1}\left[g\left(X^i_{\tau^i_D}\right)-\int_0^{\tau^i_D}f\left(X^i_s\right)\mathrm{d}s\right]-u(x)\right)\\
&=\mathrm{Normal}\left(0,\mathrm{Var}\left(g(X_{\tau_D})-\int^{\tau_D}_0f(X_s)\mathrm{d}s\right)\right).
\end{aligned}
\end{equation}
From (\ref{eq:4.15}), it can be seen that the truncation error is $O(1/\sqrt{n})$ for the Monte Carlo method. Or rather, the error is approximately a normal random variable for large $n$, i.e.,
\begin{equation}\label{eq:4.17}
\mathrm{error}\approx  \widetilde{X}/\sqrt{n},
\end{equation}
where $\widetilde{X}$ is a normal random variable with the distribution Eq. \eqref{eq:4.15}. One can reduce the error by increasing $n$.

\section{Numerical experiments}
In this section, based on (\ref{eq:4.12}), we numerically solve Eq. \eqref{eq:1.4} by generating the paths of the stochastic processes $X_{\tau_D}$. The validity of the numerical method is verified by comparing the simulation result with the exact solution. 

In the simulation, the parameters are taken as follows. The domain $D$ is the unit ball in $\mathrm{R}^2$, $f(x)=0$, $g(x)=x_1+x_2$ for $x=[x_1,x_2]\notin D$, and $X_0=[-0.2,0.9]$. The probability distribution of particles in direction $m(\theta)=1/{\pi}$ for $\arg(\theta)\in(0.5\pi,\ \pi)$ and $m(\theta)=1/{3\pi}$ for $\arg(\theta)\in(0,\ 0.5\pi)\cup(\pi,\ 2\pi)$. Then, according to Eq.\ \eqref{eq:1.5}, we obtain the exact solution of Eq.\ \eqref{eq:1.4}, that is, $u(x)=x_1+x_2$; in particular,  $u(X_0)=x_1+x_2=0.7$. Then, using Eq.\ \eqref{eq:4.12}, one can compute the numerical solution of Eq.\ \eqref{eq:1.4}. 
For the algorithm of simulation (see Appendix), we take the sample number $n=10000$, $\Delta t=5\times10^{-4}$, $b=10$, and  $c_{m, \alpha}=\frac{1}{|\Gamma(-\alpha)|}$. The above functions and parameters remain unchanged unless otherwise specified.

\begin{figure}[H]
\centering
\subfigure[]{
\begin{minipage}[t]{0.4\textwidth}\label{fig:a2}
\centering
\includegraphics[scale=0.35]{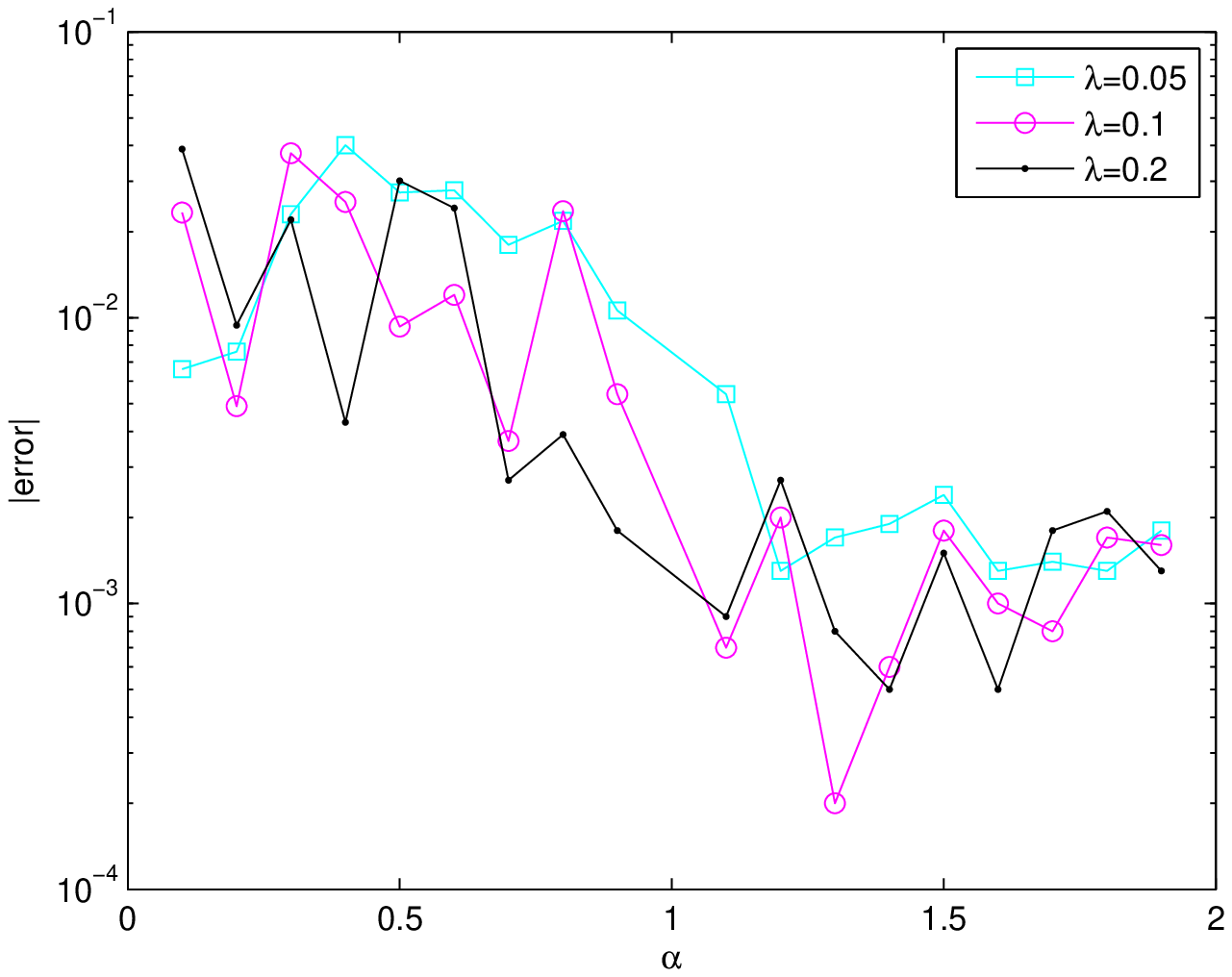}
\end{minipage}}\subfigure[]{
\begin{minipage}[t]{0.4\textwidth}\label{fig:b2}
\centering
\includegraphics[scale=0.35]{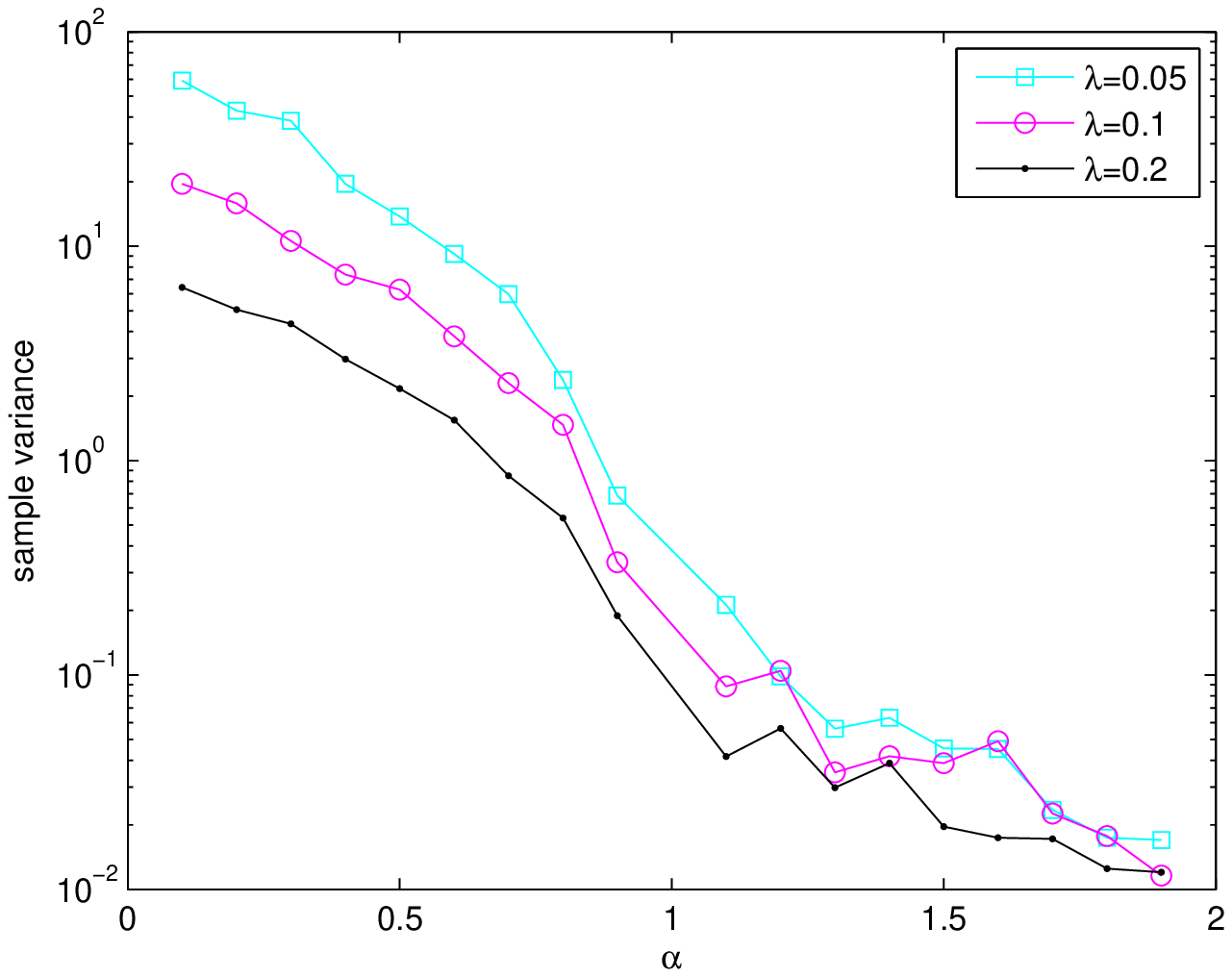}
\end{minipage}}
\caption{Simulation results for Eq.\ \eqref{eq:1.4}. The left-hand plot (a) is for $|\mathrm{error}|$ (\ref{Numer_error}), and the right-hand plot (b) for the sample variance.}\label{fig:3}
\end{figure}
\figurename\ \ref{fig:3} shows that the sample variances decrease with the increase of $\alpha$ and  $ \lambda$, and similarly $|\mathrm{error}|$ also tends to decrease. This figure also illustrates  the effect of variance on the $\mathrm{error}$. Next, we show the influence of $n$ on the error.

%
\begin{figure}[H]
\centering
\subfigure[]{
\begin{minipage}[t]{0.4\textwidth}\label{fig:a2}
\centering
\includegraphics[scale=0.35]{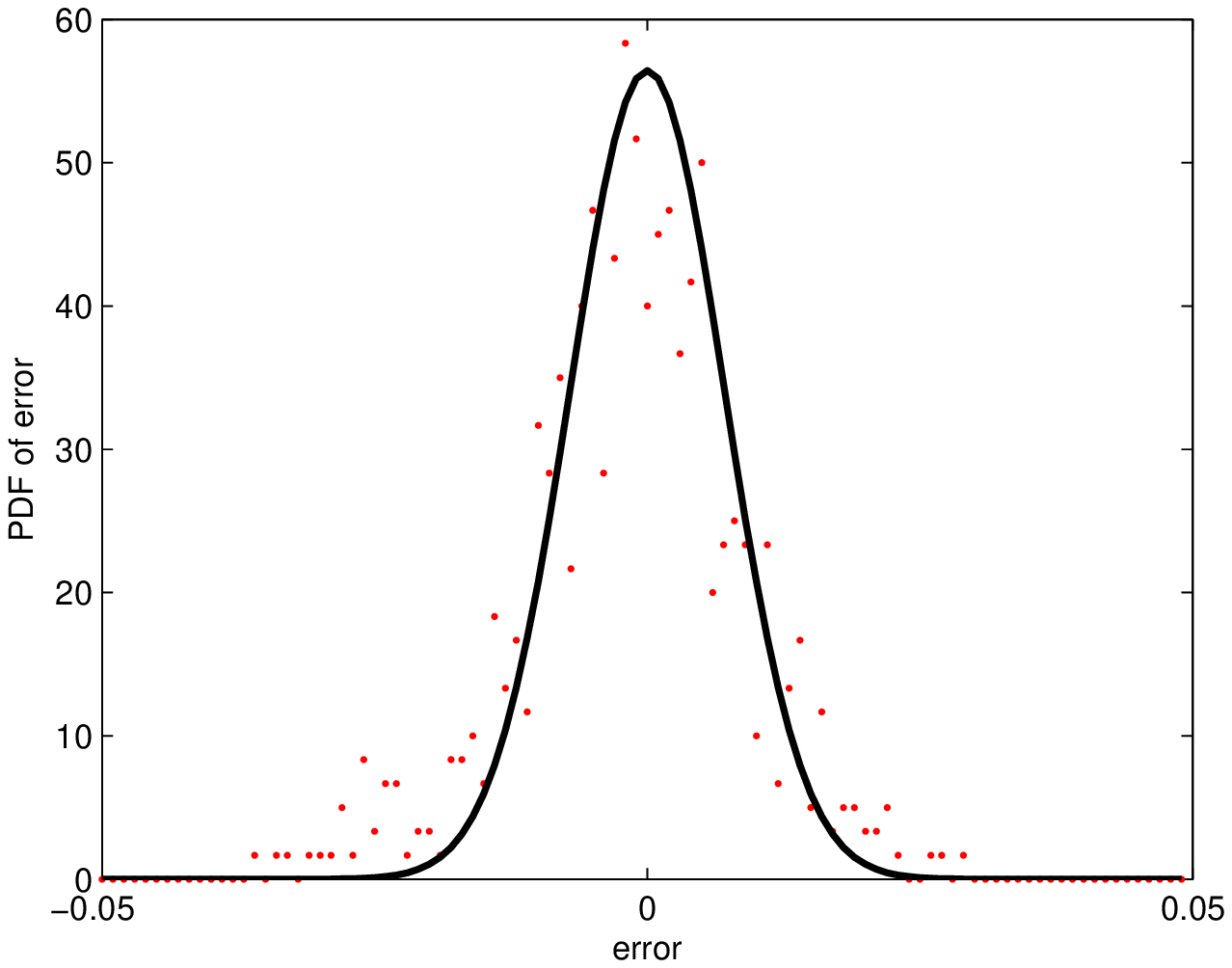}
\end{minipage}}\subfigure[]{
\begin{minipage}[t]{0.4\textwidth}\label{fig:b2}
\centering
\includegraphics[scale=0.35]{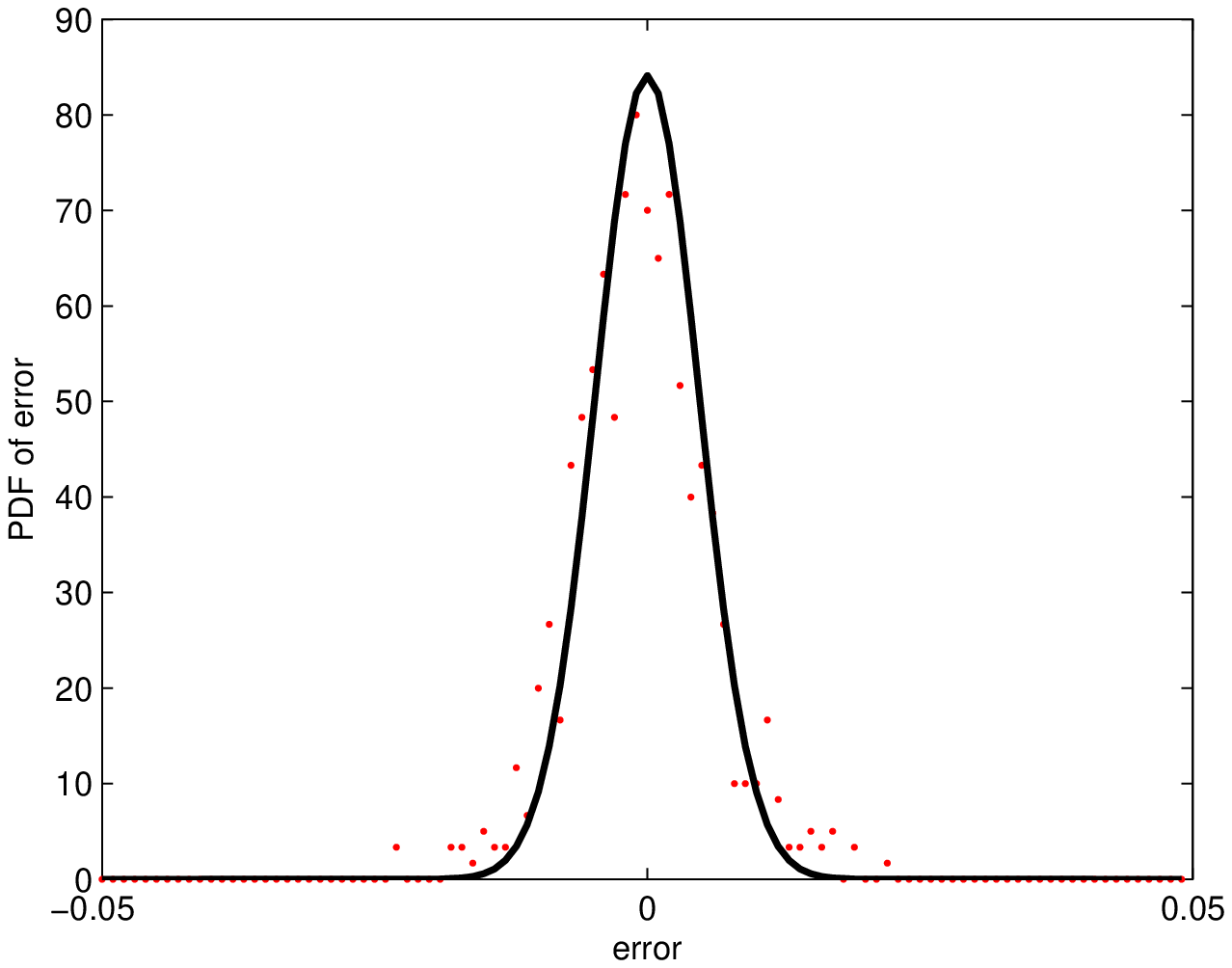}
\end{minipage}}
\caption{Distribution of simulation errors for $\alpha=1.2$, $\lambda=0.05$, and $n=2000$ for (a) ($n=4000$ for (b)).  
}\label{fig:4}
\end{figure}
For fixed $n$, repeating the simulation $600$ times leads to the approximate distribution of errors. Figure \ \ref{fig:4} shows that the errors are normally distributed, where the real curve is the plot of the function $\left(2\pi\sigma^2/n\right)^{-0.5}\exp\left(-\frac{x^2}{2\sigma^2/n}\right)$ with $\sigma^2=0.1$ obtained from \figurename \ \ref{fig:3}. Obviously the larger $n$ is, the smaller the variance of errors becomes.
Figure \ref{fig:5} indicates the convergence of the algorithm, as expected, being  $O(1/\sqrt{n})$.
\captionsetup[figure]{name={Fig.}}
\begin{figure}[H]
\centering
\includegraphics[scale=0.45]{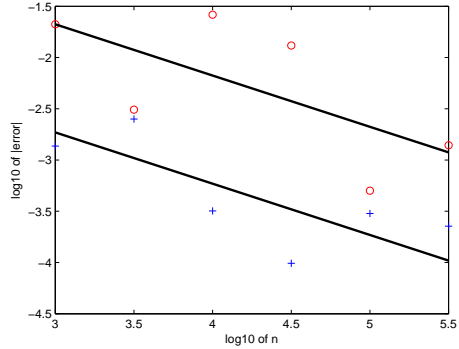}
\caption{ Convergence of the simulations, respectively, with  $\alpha=0.7, \lambda=0.2$ (red circle) and $\alpha=1.2, \lambda=0.05$ (blue plus ), when $n$ increases.
}\label{fig:5}
\end{figure}

\section{Conclusion}
The first exit and Dirichlet problems for the nonisotropic tempered $\alpha$-stable process $X_t$ have been discussed. With the obtained upper bounds of all moments of the first exit position $\left|X_{\tau_D}\right|$ and the first exit time $\tau_D$, we show that the PDF of $\left|X_{\tau_D}\right|$ or $\tau_D$ exponentially decays with the increase of $\left|X_{\tau_D}\right|$ or $\tau_D$, and $\mathrm{E}\left[\tau_D\right]\sim \left|\mathrm{E}\left[X_{\tau_D}\right]\right|$,\ $\mathrm{E}\left[\tau_D\right]\sim\mathrm{E}\left[\left|X_{\tau_D}-\mathrm{E}\left[X_{\tau_D}\right]\right|^2\right] $. The Feynman-Kac representation is provided for the Dirichlet problem with the operator $\mathrm{\Delta}^{\alpha/2,\lambda}_m$, and some numerical simulations are performed to show its usefulness.

\section*{Acknowledgements}
This work was supported by the National Natural Science Foundation of China under grant no. 11671182, and the Fundamental Research Funds for the Central Universities under grant no. lzujbky-2018-ot03.

\appendix
\section{Description for the algorithm of simulation  }
 We work in two dimensions. Let $m(\theta)$ be the probability distribution of particles in $\theta$-direction, and $c_{m, \alpha}=\frac{1}{|\Gamma(-\alpha)|}$. Referring to \cite{48}, we present the description of the algorithm.

For $0<\alpha<1$, set
\begin{equation}\label{eq:A1}
S=\left(\Delta t\right)^{1/\alpha}\frac{\sin \alpha(U+\pi/2)}{\cos(U)^{1/\alpha}}\left(\frac{\cos\left(U-\alpha(U+\pi/2)\right)}{W}\right)^{(1-\alpha)\alpha},
\end{equation}
where $U$ is an uniform distribution on $[-\pi/2,\pi/2]$, and $W$ is an exponential distribution with mean 1. Generate the random variable (r.v.) $Z$ of exponential distribution with mean $\lambda^{-1}$; 
 if $Z<S$, reject and draw again, otherwise set $X_{\Delta t}=[S\cos\theta,S\sin\theta]$, where  the r.v. $\theta$ is generated by the PDF $m(\theta)$.

When $1<\alpha<2$, set
\begin{equation}\label{eq:A2}
S=\left(\Delta t\right)^{1/\alpha}\frac{\sin \alpha(U-\pi/2)}{\cos(U)^{1/\alpha}}\left(\frac{\cos\left(U-\alpha(U-\pi/2+\pi/\alpha)\right)}{W}\right)^{(1-\alpha)\alpha};
\end{equation}
if $S>Z-b\,(b>0)$, reject and draw again, otherwise set $X_{\Delta t}=[S\cos\theta,S\sin\theta]$; again the PDF of $\theta$ is $m(\theta)$.

To simulate the entire path of the stable process, one can rewrite $X_t$ as follows
\begin{equation*}
X_t=\sum^{t/{\Delta t}}_{i=1}\left[X_{i\Delta t}-X_{(i-1)\Delta t}\right].
\end{equation*}
The stationary and independent increments of $X_t$ show that
\begin{equation*}
X_{i\Delta t}-X_{(i-1)\Delta t}\stackrel{d}{=}X_{\Delta t}.
\end{equation*}
According to the above, one can generate the stochastic processes $X^j_t\,(j=1,2,\dots,n)$, which denotes the path of the $j$-th particle.

To calculate the PDF of $\tau_D$, divide the time interval $[0, T]$ into $m_1$ equal parts, i.e., $0=t_0<t_1<\dots<t_{m_1}=T, \,t_i=i h_1\,(i=0,1,\cdots,m_1)$. Count the number $n_1^{i+1}$ of particles, the time of which spend on lies in the interval $(t_i,t_{i+1}]$  when firstly leaving the domain $B(0,r)$. Then, $\frac{n_1^{i+1}}{nh_1}$ denotes the PDF of $\tau_D$ in  $(t_i,t_{i+1}]$.


To calculate the PDF of $|X_{\tau_D}|$, divide the interval $(r,l]$ into $m_2$ equal parts, i.e.,  $r=r_0<r_1<\cdots<r_{m_2}=l, r_i=ih_2+r_0\,(i=0,1,\cdots,m_2)$. Count the number $n_2^{i+1}$ of particles that fall into the annular region $(r_i,r_{i+1}]$, when first exiting the domain $B(0,r)$. Then, $\frac{n_2^{i+1}}{nh_2}$ denotes the PDF of $|X_{\tau_D}|$ in $(r_i,r_{i+1}]$.

%
%

\end{document}